\documentclass[12pt, leqno]{amsart}
\usepackage[active]{srcltx}
\usepackage{verbatim}
\usepackage{epsfig,graphicx,color,mathrsfs}
\usepackage{graphicx}
\usepackage{amsmath,amssymb,amsthm,amsfonts}
\usepackage{amssymb}
\usepackage[english]{babel}
\usepackage{appendix}

\usepackage[left=2.7cm,right=2.7cm,top=3cm,bottom=3cm]{geometry}

\newcommand{\R}{{\mathbb R}}

\newcommand{\rn}{{\mathbb{R}^N}}

\numberwithin{equation}{section}
\newtheorem{theorem}{Theorem}[section]
\newtheorem{proposition}[theorem]{Proposition}
\newtheorem{lemma}[theorem]{Lemma}

\newtheorem{definition}[theorem]{Definition}
\newtheorem{remark}[theorem]{Remark}
\theoremstyle{definition}

\newcommand{\brm}{\begin{remark}\rm}
\newcommand{\erm}{\end{remark}}
\newcommand{\brms}{\begin{remark}\rm}
\newcommand{\erms}{\end{remark}}
\newcommand{\bte}{\begin{theorem}}
\newcommand{\ete}{\end{theorem}}
\newcommand{\bpr}{\begin{proposition}}
\newcommand{\epr}{\end{proposition}}
\newcommand{\ble}{\begin{lemma}}
\newcommand{\ele}{\end{lemma}}
\newcommand{\beq}{\begin{equation}}
\newcommand{\eeq}{\end{equation}}
\newcommand{\bdm}{\begin{displaymath}}
\newcommand{\edm}{\end{displaymath}}
\numberwithin{equation}{section}

\newcommand{\bos}{\begin{remark}\rm}
\newcommand{\eos}{\end{remark}}

\newcommand{\ben}{\begin{enumerate}}
\newcommand{\een}{\end{enumerate}}

\newcommand{\e }{\varepsilon }

\newcommand{\be}{\begin{equation}}
\newcommand{\ee}{\end{equation}}

\newcommand{\inn}{\text{  in   }}

\title[Nonlocal singular problems]{A  variational approach to nonlocal singular problems}

\author[A. Canino]{Annamaria Canino$^*$}

\author[L.\ Montoro]{Luigi Montoro$^*$}

\author[B.\ Sciunzi]{Berardino Sciunzi$^*$}

\keywords{Fractional Laplacian, singular problems, variational methods.}

\thanks{\it 2010 Mathematics Subject
 Classification: 	35S05, 35S15, 35R11, 60G22}

\thanks{$^*$Dipartimento di Matematica e Informatica,
Universit\`a della Calabria,
Ponte Pietro Bucci 31B, I-87036 Arcavacata di Rende, Cosenza, Italy. E-mail: {\em canino@mat.unical.it},  {\em montoro@mat.unical.it}, {\em sciunzi@mat.unical.it}}

\thanks{The authors are members of the Gruppo Nazionale per l'Analisi Matematica, la Probabilit\`a e le loro Applicazioni
(GNAMPA) of the Istituto Nazionale di Alta Matematica (INdAM)}

\begin{document}
\begin{abstract}
We provide a suitable  variational approach
for  a class of nonlocal problems  involving the fractional laplacian and  singular nonlinearities for which  the standard techniques fail.
As a corollary we deduce a characterization of the solutions.
\end{abstract}

\maketitle

\medskip

\section{Introduction and results}\label{introdue}

{In recent years, considerable attention has been given to equations involving general integrodifferential
operators, especially, those with the fractional Laplacian operator. This is related 
to the fact that the nonlocal structure has connection with many real world phenomena. Indeed, nonlocal operators naturally appear in elasticity problems \cite{signorini}, thin obstacle problem \cite{Caf79}, phase transition \cite{AB98, CSM05, SV08b}, flames propagation \cite{crs}, crystal dislocation \cite{fdpv, toland}, stratified materials \cite{savin_vald}, quasi-geostrophic flows \cite{caf_vasseur} and others. Since these operators are also related to L\'{e}vy processes and have a lot of applications to mathematical finance, they have been also studied from a probabilistic point of view (see for example \cite{kass3, bogdan_by, kass1, ito, enrico}).}  We refer the readers to, for instance,  \cite{BMS, BjCaffFigalli,  CafFigalli,  CScpam, CS2011,  CS2012, CMoS, guida, DMPS,  quim} where existence of solutions, qualitative properties of solutions and regularity of solutions  are studied for some nonlocal problems.
In this paper we aim to  provide  a variational structure to the following problem
\begin{equation}\tag{$\mathcal P_{\gamma}$}\label{Eq:P}
\left\{
\begin{array}{rcl}
(-\Delta)^s u&=&{\displaystyle \frac{1}{u^\gamma}}+\omega\quad\text{in}\,\, \Omega,\\
u&>&0\,\qquad\quad\,\inn \Omega,\\
u&=&0\,\qquad\quad\,\inn \rn\setminus\Omega,\\
\end{array}
\right.
\end{equation}
where $\omega\in W^{-s,2}(\Omega)$ i.e. the dual space of $W^{s,2}_0(\Omega)$ that we will define below,   $\Omega$ is a bounded smooth domain, $0<s<1$, $N>2s$, $(-\Delta)^s$ is the fractional Laplacian (see Section \ref{sec: 2} for the definition) and the equation is understood as in Definition \ref{def:definsol}.

In spite of the fact that \eqref{Eq:P} is formally the Euler equation of the functional
\[J(u)=\frac{c_{N,s}}{4}\int_{\mathbb{R}^{2N}} \frac{(u(x) - u(y))^2}{|x - y|^{N+2s}}\,dx\,dy+\int_{\Omega}\Phi(u)\, dx-\langle\omega, u \rangle \quad u\in W^{s,2}_0(\Omega),\]
where
\begin{equation}\label{eq:sssstuuuu}
\Phi(s)=
\begin{cases}
{\displaystyle-\int_1^st^{-\gamma}\,dt} & \text{if}\,s\geq0\\
+\infty& \text{if}\,s<0,
\end{cases}
\end{equation}
a standard variational approach is obstructed by the fact that the energy functional might be identically infinity as it is the case when solutions  do not belong to $W^{s,2}_0(\Omega)$.  Even in  the local case it has been shown in \cite{crandall,LM}
that the solution cannot belong to $W^{1,2}_0(\Omega)$ if $\gamma\geq 3$ so that, as remarked, the classical approach cannot be exploited.
However  many results have been obtained in the literature developing alternative techniques. We only mention here the related results in
 \cite{boccardo,C23,CanDeg,CGS,kaw,lair,LM,saccon,stuart}.

The study of nonlocal problems involving singular nonlinearities is quite undertaken in the literature.
Existence and
 uniqueness of the solution to \eqref{Eq:P} where studied in the recent works \cite{bbmp, CMoSS}. 
 Here, to deal with the nonlocal case,  we exploit some ideas  introduced in \cite{CanDeg} facing the difficulties caused by the nonlocal nature of the problem. In all the paper  we shall take into account the fact that the solutions are not in the classical nonlocal Sobolev spaces and the boundary datum has to be understood in a nonstandard way.

Let us now state our main result.
\begin{theorem}\label{main}
Let $\gamma>0$, $\omega\in W^{-s,2}(\Omega)$ and $u\in W^{s,2}_{\operatorname{loc}}(\Omega)$. If $u$ satisfies the problem
\begin{equation}\tag{$\mathcal P_{D}$}\label{eq:amkdigola1}
\begin{cases}
u>0\,\, \text{a.e. in}\,\,\Omega\,\, \text{and} \,\, u^{-\gamma} \in L^1_{\operatorname{loc}}(\Omega),\\
\\
{\displaystyle \iint_{\mathbb R^{2N}}\frac{\big(u(x)-u(y)\big)(\varphi(x)-\varphi(y))}{|x-y|^{N+2s}}\,dx\,dy-\int_{\Omega}u^{-\gamma}\varphi\,dx= \langle \omega,\varphi\rangle}\quad \forall \varphi\in C^{\infty}_{c}(\Omega),\\
\\
u\leq 0 \,\, \text{on}\,\, \partial \Omega,
\end{cases}
\end{equation}
then $u$ is the solution to the problem
\begin{equation}\tag{$\mathcal P_{V}$}\label{eq:amkdigola2}
\begin{cases}
u>0\,\, \text{a.e. in}\,\,\Omega\,\, \text{and} \,\, u^{-\gamma} \in L^1_{\operatorname{loc}}(\Omega),\\
\\
{\displaystyle \iint_{\mathbb R^{2N}}\frac{\big(u(x)-u(y)\big)\big((v(x)-u(x))-(v(y)-u(y))\big)}{|x-y|^{N+2s}}\,dx\,dy-\int_{\Omega}u^{-\gamma}(v-u)\,dx\geq \langle \omega, v-u \rangle }\\
\,\quad\qquad\qquad\qquad\qquad \qquad\qquad\qquad\qquad\forall v\in u+(W^{s,2}_0(\Omega)\cap L^{\infty}_{c}(\Omega))\,\, \text{with}\,\, v\geq 0 \,\, \text{a.e.}\,\, \text{in}\,\, \Omega,\\
\\
u\leq 0 \,\, \text{on}\,\, \partial \Omega.
\end{cases}
\end{equation}
Moreover if $\omega\in W^{-s,2}(\Omega)\cap L^1_{\operatorname{loc}}(\Omega)$, then the problems \eqref{eq:amkdigola1} and \eqref{eq:amkdigola2} are equivalent.
\end{theorem}
Note that $u\in W^{s,2}_{\operatorname{loc}}(\Omega)$ is a solution to \eqref{eq:amkdigola2} if and only if $u$ is the minimum of a suitable functional actually defined in \eqref{eq:ognivoltanottevasco}. Remarkably, this provides a variational characterization of the solutions that is completely new in this setting and that could be exploited to deduce existence and multiplicity results under suitable assumptions.\\

\noindent Furthermore,
as consequence  of  Theorem \ref{main}, we also provide a decomposition of the solution $u$. Namely we deduce that:
\[
u=u_0+w
\]
where, $u_0\in L^{\infty}(\Omega)$ is the unique solution to \eqref{Eq:P} with $\omega \equiv0$ (see Proposition \ref{pro:u_0} below) and
$w\in W^{s,2}_{0}(\Omega)$ is a critical point (in the meaning of \cite{Szul} ) of an associated functional.

\noindent To state such a result let us start considering
$g: \Omega \times \mathbb{R} \rightarrow \mathbb{R}$ satisfying the growth assumption
\begin{equation}\label{eq:hypong}
\begin{cases}
\text{there exists}\,\,  a \in L^{\frac{2N}{N+2s}}(\Omega) \,\, \text{and} \,\, b\in \mathbb{R}\,\, \text{such that}\\
 |g(x,t)| \le a(x) + b\,|t|^{\frac{N+2s}{N-2s}}  \,\, \text{for a.e.}\,\, x \in \Omega \text{ and every } t \in \mathbb{R}.
\end{cases}
\end{equation}
Then let
$g_1(x,t) = g(x, u_0(x) + t)$,
$G_1(x,t)=\int_0^t g_1(x,t)
\, dt$ and   $\Phi : W^{s,2}_0(\Omega)
\,\rightarrow \, \rbrack -\infty, +\infty\lbrack$ the $C^1$ functional defined by
$$\Phi(u) =- \int_{\Omega}G_1(x, u)\,dt. $$
Moreover let $\Psi : W^{s,2}_0(\Omega)
\,\rightarrow \, \rbrack -\infty, +\infty]$ be the convex functional defied by
\[
\Psi(v)=\frac{c_{N,s}}{4}\iint_{\mathbb{R}^{2N}} \frac{|v(x)- v(y)|^2}{|x - y|^{N+2s}}\,dx\,dy
+\int_{\Omega}G_0(x,v)\, dx.
\]
Finally define $F: W^{s,2}_0(\Omega)
\,\rightarrow \, \rbrack -\infty, +\infty]$  by
\begin{equation}\label{eq:miseravallafinemannaia}
F(v)=\Psi(v)+ \Phi(v).
\end{equation}
We have the following
\begin{theorem}\label{eq:thspostato}
Let $\gamma>0$.

The function $u\in W^{s,2}_{\operatorname{loc}}(\Omega)\cap L^{\frac{2N}{N-2s}(\Omega)}$ is a solution to the problem
\begin{equation}\label{eq:amkdigola11estioam}
\begin{cases}
u>0\,\, \text{a.e. in}\,\,\Omega\,\, \text{and} \,\, u^{-\gamma} \in L^1_{\operatorname{loc}}(\Omega),\\
\\
{\displaystyle \iint_{\mathbb R^{2N}}\frac{\big(u(x)-u(y)\big)(\varphi(x)-\varphi(y))}{|x-y|^{N+2s}}\,dx\,dy=\int_{\Omega}u^{-\gamma}\varphi\,dx +\int_{\Omega}g(x,u)\varphi\,dx\quad \forall \varphi\in C^{\infty}_{c}(\Omega)},\\
\\
u\leq 0 \,\, \text{on}\,\, \partial \Omega,
\end{cases}
\end{equation}

\

if and only if

\
\begin{equation}\label{eq:lastappl}
u\in u_0+W^{s,2}_0(\Omega)\,\,\text{and}\,\, w:=u-u_0 \,\, \text{is a critical point of } F.
\end{equation}
\end{theorem}

\

\noindent The paper is organized as follows: in Section \ref{sec: 2} we give some preliminaries related to the functional framework associated to problem \eqref{Eq:P},  we introduce the proper notion of solution that will be used through this work and some preliminary results. Section \ref{sec: main} deals with the proof of the main result of this work.

\section{Notations and Preliminary Results}\label{sec: 2}
\noindent Let us recall that, given  a function $u$ in the Schwartz's class $\mathcal{S}(\mathbb{R}^{N})$ we define for $0<s<1$, the fractional Laplacian as
\begin{equation}\label{fourier1}
\widehat{(-\Delta)^{s}}u(\xi)=|\xi|^{2s}\widehat{u}(\xi),\quad \xi\in\mathbb{R}^{N},
\end{equation}
where $\widehat u\equiv \mathfrak F(u)$ is the Fourier transform of $u$.
It is well known (see for example \cite{Stein, enrico}) that this operator can be also represented, for suitable functions, as a principal value of the form
\begin{equation}\label{org1}
(-\Delta)^s u(x):=c_{N,s}\,{\rm P.V.}\int_{\mathbb{R}^N}\frac{u(x)-u(y)}{|x-y|^{N+2s}}\,dy
\end{equation}
where
\begin{equation}\label{constante}
c_{N,s}:=\left(\int_{\mathbb{R}^{N}}{\frac{1-\cos(\xi_1)}{|\xi|^{N+2s}}\, d\xi}\right)^{-1}=\frac{4^s\Gamma\left(\frac{N}{2}+s\right)}{-\pi^{\frac{N}{2}}\Gamma(-s)}>0,
\end{equation}
is a normalizing constant chosen to guarantee that \eqref{fourier1} is satisfied (see \cite{guida, PhS, enrico}).

The symbol $\|\cdot\|_{L^p(\Omega)}$ stands for the standard norm for the $L^p(\Omega)$ space. For a measurable function $u:\mathbb{R}^N\to\mathbb{R}$, we let
\[
[u]_{D^{s,2}(\mathbb{R}^N)} := \left(\int_{\mathbb{R}^{2N}} \frac{|u(x) - u(y)|^2}{|x - y|^{N+2s}}\, dx dy\right)^{1/2}
\]
be its Gagliardo seminorm. We consider the space
$$
W^{s,2}(\mathbb{R}^N) := \big\{u \in L^{2}(\mathbb{R}^N) : [u]_{D^{s,2}(\mathbb{R}^N)}<\infty\big\},
$$
endowed with norm $\|\,\cdot\,\|_{L^2(\mathbb{R}^N)}+[\,\cdot\,]_{D^{s,2}(\mathbb{R}^N)}$. For $\Omega\subset\mathbb{R}^N$ open and bounded,
we consider
\[
W^{s,2}_0(\Omega) := \big\{u \in W^{s,2}(\mathbb{R}^N): \text{$u=0$ a.e. in $\mathbb{R}^N\setminus\Omega$}\big\},
\]
endowed with norm $[\,\cdot\,]_{D^{s,2}(\mathbb{R}^N)}$.
The imbedding $W^{s,2}_0(\Omega)\hookrightarrow L^r(\Omega)$ is continuous
for $1\le r \le 2_s^\ast$ and compact for $1\le r< 2_s^\ast$, where  $2^*_s:={2N}/{(N-2s)}$ and $N>2s$ (as we are assuming throughout the paper).
The space $W^{s,2}_0(\Omega)$ can be equivalently defined as the completion
of $C^\infty_0(\Omega)$
in the norm $\|\,\cdot\,\|_{L^2(\mathbb{R}^N)}+[\,\cdot\,]_{D^{s,2}(\mathbb{R}^N)}$, provided $\partial\Omega$ is smooth enough. In this context by $C^\infty_0(\Omega)$ we mean the space
\begin{equation}\nonumber
C^\infty_0(\Omega):=\{f:\mathbb{R}^N\rightarrow \mathbb{R}\,:\, f\in C^\infty(\mathbb{R}^N), \, \text{support $f$ is compact  and support $f \subseteq\Omega$} \}.
\end{equation}
We shall denote the localized Gagliardo seminorm by
$$
[u]_{W^{s,2}(\Omega)}:
=\left(\int_{\Omega\times\Omega} \frac{|u(x)-u(y)|^2}{|x-y|^{N+2s}}\,dx\,dy\right)^{1/2}.
$$
Finally define the  space
$$
W^{s,2}_{{\rm loc}}(\Omega) := \big\{u:\Omega\to\R : u \in L^{2}(K),\,\, [u]_{W^{s,2}(K)}<\infty, \,\,\,\text{for all $K\Subset\Omega$}\big\}.
$$

Since the way of understanding the boundary condition is not unambiguous, we  give the following (yet introduced in \cite{CMoSS}):

\begin{definition}\label{rediri}
We say
that $u\leq 0$ on $\partial\Omega$ if, for every $\e>0$, it follows that
\[
(u-\e)^+\in W^{s,2}_0(\Omega)\,.
\]
We will say that $u= 0$ on $\partial\Omega$ if $u$ is non-negative  and $u\leq 0$ on $\partial\Omega$.
\end{definition}

First of all, in order to give a weak formulation to the problem \eqref{Eq:P}, we prove the following
\begin{proposition}\label{pr:welldefined}
Let $u\in W^{s,2}_{\operatorname{loc}}(\Omega)\cap L^{1}(\Omega)$ and $u=0$ for a.e. $x\in \mathbb{R}^N\setminus \Omega$. Then for any $\varphi \in C^\infty_c(\Omega)$,
\[ \frac 12 c_{N,s}\iint_{\mathbb{R}^{2N}}\frac{(u(x)-u(y))(\varphi(x)-\varphi(y))}{|x-y|^{N+2s}}\,dx\,dy <\infty\,.
\]
\end{proposition}
\begin{proof}
Let  $\varphi \in C^\infty_c(\Omega)$ and let us denote $K_{\varphi}= \operatorname{supp}(\varphi) $. Fix now a compact set $ K\subset \Omega$ such that $K_\varphi\subset K$ and use the decomposition
\begin{equation}\nonumber
\begin{split}
\mathbb{R}^N\times \mathbb{R}^N\,=\,\left( K\cup K^c\right)\times \left( K\cup K^c\right),
\end{split}
\end{equation}
where $K^c:= \mathbb R^N \setminus K$. Thus
\begin{eqnarray}\label{eq:split}
&&\frac 12 c_{N,s}\iint_{\mathbb{R}^{2N}}\frac{(u(x)-u(y))(\varphi(x)-\varphi(y))}{|x-y|^{N+2s}}\,dx\,dy\\\nonumber&&= \frac 12 c_{N,s}\iint_{K\times K}\frac{(u(x)-u(y))(\varphi(x)-\varphi(y))}{|x-y|^{N+2s}}\,dx\,dy \\\nonumber
&&+\frac 12 c_{N,s}\iint_{K\times K^c}\frac{(u(x)-u(y))(\varphi(x)-\varphi(y))}{|x-y|^{N+2s}}\,dx\,dy
 \\\nonumber
&&+\frac 12 c_{N,s}\iint_{K^c\times K}\frac{(u(x)-u(y))(\varphi(x)-\varphi(y))}{|x-y|^{N+2s}}\,dx\,dy,
\end{eqnarray}
since
\begin{equation}\nonumber
\frac 12 c_{N,s}\iint_{K^c \times K^c}\frac{(u(x)-u(y))(\varphi(x)-\varphi(y))}{|x-y|^{N+2s}}\,dx\,dy=0.
\end{equation}
We prove that all the three terms on the right-hand side of \eqref{eq:split} are well defined. In fact
\begin{equation}\nonumber
\frac 12 c_{N,s}\iint_{K \times K}\frac{(u(x)-u(y))(\varphi(x)-\varphi(y))}{|x-y|^{N+2s}}\,dx\,dy<C,
\end{equation}
for some positive constant $C$, since by hypothesis $u\in W^{s,2}_{\operatorname{loc}}(\Omega)$.

We can write the second term as
\begin{eqnarray}\label{eq:split2}
&&\frac 12 c_{N,s}\iint_{K\times K^c}\frac{(u(x)-u(y))(\varphi(x)-\varphi(y))}{|x-y|^{N+2s}}\,dx\,dy\\\nonumber
&&=\frac 12 c_{N,s}\iint_{K_{\varphi}\times K^c}\frac{(u(x)-u(y))(\varphi(x)-\varphi(y))}{|x-y|^{N+2s}}\,dx\,dy.
\end{eqnarray}
We observe that,  for all points $(x,y)\in K_\varphi\times K^c$, we have that $|x-y|\geq \delta>0$, for some positive constant $\delta=\delta(K,K_\varphi)$ and therefore
\begin{eqnarray}\label{eq:split22}
&&\frac 12 c_{N,s}\iint_{K_\varphi \times K^c}\frac{(u(x)-u(y))(\varphi(x)-\varphi(y))}{|x-y|^{N+2s}}\,dx\,dy\\\nonumber
&&\leq C,
\end{eqnarray}
with $C=C(\delta,K,K_\varphi, \|u\|_{L^1(\mathbb R^N)}, \|\varphi\|_{L^{\infty}(K_{\varphi})})$ a positive constant. Here  we have used the fact  that  $u\in L^1(\mathbb R^N)$ (since $u\in L^1(\Omega)$ and $u=0$ a.e. in $\mathbb{R}^N\setminus \Omega)$  and $\varphi \in C^\infty(K_{\varphi})$.
>From \eqref{eq:split2} and \eqref{eq:split22} we obtain
\begin{equation}\label{eq:split222}
\frac 12 c_{N,s}\iint_{K \times K^c}\frac{(u(x)-u(y))(\varphi(x)-\varphi(y))}{|x-y|^{N+2s}}\,dx\,dy\leq C.
\end{equation}
For the third term we argue in the same way as  in \eqref{eq:split2}, \eqref{eq:split22} and \eqref{eq:split222}.  Finally, by \eqref{eq:split} we obtain the thesis.
\end{proof}
Having in mind Proposition \ref{pr:welldefined} , the basic definition of solution can be formulated in the following
\begin{definition}\label{def:definsol} A positive function $u \in W^{s,2}_{{\rm loc}}(\Omega)\cap L^{1}(\Omega)$ is a weak solution to problem~\eqref{Eq:P} if ${u^{-\gamma}} \in L^{1}_{{\rm loc}}(\Omega)$,
\[u=0\qquad   \text{for a.e.}\,\, x\in \mathbb{R}^N\setminus \Omega\]
and we have
\begin{equation}\nonumber
		\frac{c_{N,s}}{2}\iint_{\mathbb{R}^{2N}} \frac{(u(x) - u(y))\, (\varphi(x) - \varphi(y))}{|x - y|^{N+2s}}\, dx\, dy
		=\int_{\Omega} \frac{\varphi}{u^\gamma}\, dx+\langle \omega,\varphi\rangle,
	\end{equation}
	for every $\varphi \in C^\infty_c(\Omega)$.
\end{definition}
We state a \emph{weak comparison principle} for sub-super solutions to \eqref{Eq:P}. To do this we first give the following
\begin{definition}\label{eq:subsupersolunonloc}
Given $z\in W^{s,2}_{{\rm loc}}(\Omega)\cap L^{1}(\Omega)$ with   $z\geq 0$, we say that $z$ is a weak supersolution  (respectively  subsolution)  to
\eqref{Eq:P}, if
\begin{eqnarray}\nonumber
&&\iint_{\mathbb{R}^{2N}} \frac{(z(x) - z(y))\, (\varphi(x) - \varphi(y))}{|x - y|^{N+2s}}\, dx\, dy{\geq}\int_{\Omega} \frac{\varphi}{z^\gamma}\, dx+\langle \omega,\varphi \rangle,
\\\nonumber
&&\qquad\qquad\qquad\qquad\qquad\qquad\qquad\qquad\qquad {\forall \varphi\in C^\infty_c(\Omega)\,,\, \varphi\geq 0\,}
\\\nonumber\text{(and respectively)}\\\nonumber
&&
\iint_{\mathbb{R}^{2N}} \frac{(z(x) - z(y))\, (\varphi(x) - \varphi(y))}{|x - y|^{N+2s}}\, dx\, dy{\leq}\int_{\Omega} \frac{\varphi}{z^\gamma}\, dx+\langle \omega,\varphi \rangle,
\\\nonumber
&&\qquad\qquad\qquad\qquad\qquad\qquad\qquad\qquad\qquad
 {\forall \varphi\in C^\infty_c(\Omega)\,,\, 0\leq \varphi\leq z.\,}
\end{eqnarray}
\end{definition}

\begin{theorem}\label{comparison}
Let $\gamma>0$ and  $\omega\in W^{-s,2}(\Omega)$. Let $u$ be a subsolution to  \eqref{Eq:P} such that $u\leq 0$ on $\partial\Omega$ and let $v$ be a supersolution to~\eqref{Eq:P}. Then, $u\leq v$ a.e. in $\Omega$.
\end{theorem}
\begin{proof}
Theorem \ref{comparison} can be proved as \cite[Theorem 4.2]{CMoSS}.\end{proof}
\section{Proof of Theorem \ref{main}}\label{sec: main}
\begin{proposition}\label{pro:u_0}
Let us consider the problem
\begin{equation}\label{eq:u_0}
\left\{
\begin{array}{rcl}
(-\Delta)^s u&=&{\displaystyle \frac{1}{u^\gamma}}\quad\text{in}\,\, \Omega,\\
u&>&0\quad\,\,\inn \Omega,\\
u&=&0\quad\,\,\inn \rn\setminus\Omega.\\
\end{array}
\right.
\end{equation}
Then \eqref{eq:u_0} has a unique solution $u_0\in C^{\infty}(\Omega)$  (in the sense of Definition \ref{def:definsol})  such that
\begin{itemize}
\item [$(i)$] $u_0\in W^{s,2}_0(\Omega)$  if $0<\gamma\leq 1$,   with  ${\rm{ess inf}}_K\, u>0$ for any compact $K\Subset\Omega$;
\item [$(ii)$] $u_0 \in W^{s,2}_{{\rm loc}}(\Omega)\cap L^{1}(\Omega)$
	such that $u_0^{\gamma/2}\in W^{s,2}_0(\Omega)$ if  $\gamma>1$,   with  ${\rm{ess inf}}_K\, u_0>0$ for any compact $K\Subset\Omega$.
\end{itemize}
Moreover
\begin{equation}\label{eq:boundness}
\|u_1\|_{L^{\infty}(\Omega)}^{-\frac{\gamma}{\gamma+1}}u_1\leq u_0\leq ((\gamma+1)u_1)^{\frac{1}{\gamma+1}},
\end{equation}
where $u_1$ is the solution to $(-\Delta)^su=1$ in $\Omega$ and $u=0$ in $\mathbb{R}^N\setminus \Omega$. In particular $u_0\in C(\bar \Omega).$
\end{proposition}
\begin{proof}
The existence,  uniqueness  and summability properties of the solution $u_0$ to \eqref{eq:u_0} follow by \cite[Theorem 1.2, Theorem 1.6]{CMoSS}.
We have to prove \eqref{eq:boundness}.
Let us consider the unique solution $ u_1\in W^{s,2}_0(\Omega)\cap C^{\infty}(\Omega)$ to $(-\Delta)^su=1$ in $\Omega$ and $u=0$ in $\mathbb{R}^N\setminus \Omega$. In particular we have that $u_1>0$ for any compact $K\Subset\Omega$ and by standard regularity results~\cite{quim}, it follows that $u_1\in C^s(\mathbb R^N)$.
Let us define
\begin{equation}\label{eq:hatw}
\hat w=((\gamma +1)u_1)^{\frac{1}{\gamma+1}},\qquad \gamma >0.\end{equation}
We want to show that $\hat w$ is a supersolution to \eqref{eq:u_0}, namely
 \begin{equation}\label{eq:leftwww}
\frac{c_{N,s}}{2}\iint_{\mathbb{R}^{2N}} \frac{(\hat w(x) - \hat w(y))\, (\varphi(x) - \varphi(y))}{|x - y|^{N+2s}}\, dx\, dy
		\geq\int_{\Omega} \frac{\varphi}{\hat w^\gamma}\, dx,
	\end{equation}
	for every $\varphi \in C^\infty_c(\Omega)$ and $\varphi \geq0$.
By \eqref{eq:hatw} it follows that $\hat w\in W_{{\rm loc}}^{s,2}(\Omega)\cap L^1(\Omega)$ and $\hat w=0$ in $\mathbb{R}^N\setminus\Omega$. Therefore by Proposition \ref{pr:welldefined} we have that the l.h.s. of \eqref{eq:leftwww} is well defined.
Hence
\begin{eqnarray}\label{eq:djakkjdankdkhd}
&&\frac{c_{N,s}}{2}\iint_{\mathbb{R}^{2N}} \frac{(\hat w(x) - \hat w(y))\, (\varphi(x) - \varphi(y))}{|x - y|^{N+2s}}\, dx\, dy
 \\\nonumber
&&=\frac{c_{N,s}}{2}\iint_{\mathbb{R}^{2N}\cap\{\varphi(x)\geq\varphi (y)\}} \frac{(\hat w(x) - \hat w(y))\, (\varphi(x) - \varphi(y))}{|x - y|^{N+2s}}\, dx\, dy\\\nonumber
&&+\frac{c_{N,s}}{2}\iint_{\mathbb{R}^{2N}\cap\{\varphi(x)< \varphi (y)\}} \frac{(\hat w(x) - \hat w(y))\, (\varphi(x) - \varphi(y))}{|x - y|^{N+2s}}\, dx\, dy.
\end{eqnarray}
We estimate the first term on the r.h.s of \eqref{eq:djakkjdankdkhd}.  Using a convexity argument,  we deduce
\begin{eqnarray}\label{eq:vas1}
&&\frac{c_{N,s}}{2}\iint_{\mathbb{R}^{2N}\cap\{\varphi(x)\geq\varphi (y)\}}\hat w^{-\gamma}(x) \frac{(u_1(x) - u_1(y))\, (\varphi(x) - \varphi(y))}{|x - y|^{N+2s}}\, dx\, dy\\\nonumber
&&\geq \frac{c_{N,s}}{2}\iint_{\mathbb{R}^{2N}\cap\{\varphi(x)\geq\varphi (y)\}} \frac{(u_1(x) - u_1(y))\, (\hat w^{-\gamma}(x)\varphi(x) - \hat w^{-\gamma}(y)\varphi(y))}{|x - y|^{N+2s}}\, dx\, dy\\\nonumber
&& +\frac{c_{N,s}}{2}\iint_{\mathbb{R}^{2N}\cap\{\varphi(x)\geq\varphi (y)\}} \frac{(u_1(x) - u_1(y))\, (\hat w^{-\gamma}(y) - \hat w^{-\gamma}(x ))\varphi(y)}{|x - y|^{N+2s}}\, dx\, dy\\\nonumber
&&\geq  \frac{c_{N,s}}{2}\iint_{\mathbb{R}^{2N}\cap\{\varphi(x)\geq\varphi (y)\}} \frac{(u_1(x) - u_1(y))\, (\hat w^{-\gamma}(x)\varphi(x) - \hat w^{-\gamma}(y)\varphi(y))}{|x - y|^{N+2s}}\, dx\, dy.
 \end{eqnarray}
Using a similar argument we get
\begin{eqnarray}\label{eq:vas2}
&&\frac{c_{N,s}}{2}\iint_{\mathbb{R}^{2N}\cap\{\varphi(x)<\varphi (y)\}} \frac{(u_1(x) - u_1(y))\, (\hat w^{-\gamma}(x)\varphi(x) - \hat w^{-\gamma}(y)\varphi(y))}{|x - y|^{N+2s}}\, dx\, dy
\\\nonumber
&&\geq  \frac{c_{N,s}}{2}\iint_{\mathbb{R}^{2N}\cap\{\varphi(x)<\varphi (y)\}} \frac{(u_1(x) - u_1(y))\, (\hat w^{-\gamma}(x)\varphi(x) - \hat w^{-\gamma}(y)\varphi(y))}{|x - y|^{N+2s}}\, dx\, dy.
 \end{eqnarray}
>From \eqref{eq:djakkjdankdkhd}, collecting \eqref{eq:vas1} and \eqref{eq:vas2}, we deduce
\begin{eqnarray}\nonumber
&&\frac{c_{N,s}}{2}\iint_{\mathbb{R}^{2N}} \frac{(\hat w(x) - \hat w(y))\, (\varphi(x) - \varphi(y))}{|x - y|^{N+2s}}\, dx\, dy
 \\\nonumber
 && \geq  \frac{c_{N,s}}{2}\iint_{\mathbb{R}^{2N}} \frac{(u_1(x) - u_1(y))\, (\hat w^{-\gamma}(x)\varphi(x) - \hat w^{-\gamma}(y)\varphi(y))}{|x - y|^{N+2s}}\, dx\, dy\\\nonumber
 &&=\int_{\Omega}\frac{\varphi}{\hat w^{\gamma}}\, dx,
 \end{eqnarray}
 that is \eqref{eq:leftwww}. Defining
 \begin{equation}\label{eq:checkw}
\check w=\|u_1\|^{-\frac{\gamma}{\gamma+1}}_{L^{\infty}(\Omega)}u_1,\qquad \gamma >0,\end{equation}
using the weak formulation \eqref{eq:u_0}, we can prove  as well that $\check w$ is a subsolution
to \eqref{eq:u_0}, namely
\begin{equation}\nonumber
\frac{c_{N,s}}{2}\iint_{\mathbb{R}^{2N}} \frac{(\check w(x) - \check w(y))\, (\varphi(x) - \varphi(y))}{|x - y|^{N+2s}}\, dx\, dy
		\leq\int_{\Omega} \frac{\varphi}{\check w^\gamma}\, dx,
	\end{equation}
	for every $\varphi \in C^\infty_c(\Omega)$ and $\varphi \geq0$.
Then using the definitions \eqref{eq:hatw} and \eqref{eq:checkw}, together with Theorem \ref{comparison} we get \eqref{eq:boundness}. Now  it readily follows that $u_0\in C(\bar \Omega).$
\end{proof}
Let $u_0$ as in Proposition  \ref{pro:u_0}. Let $G_0: \Omega \times \mathbb R \rightarrow [0,+\infty]$ be defined by
\begin{equation}\label{eq:G0}
G_0(x,s)=\Phi(u_0(x)+s)-  \Phi(u_0(x))+s u_0(x)^{-\gamma},
\end{equation}
where $\Phi(\cdot)$ is defined in \eqref{eq:sssstuuuu}. Then $G_0(x,0)=0$ and $G_0(x,\cdot)$ is convex and lower semicontinuous for any $x\in \Omega$. Moreover $G_0(x,\cdot)$ is $C^1$ on $]-u_0(x),+\infty[$ with
\begin{equation}\label{eq:derivG}D_s G_0(x,s)=u_0^{-\gamma}(x)-(u_0(x)+s)^{-\gamma}. \end{equation}
Let us define the functional
\begin{eqnarray}\label{eq:ognivoltanottevasco}
&&J_{\omega}(u)=\frac{c_{N,s}}{4}\iint_{\mathbb{R}^{2N}} \frac{\big ((u(x)-u_0(x)) - (u(y)-u_0(y))\big)^2}{|x - y|^{N+2s}}\,dx\,dy\\\nonumber
&&+\int_{\Omega}G_0(x,u-u_0)\, dx- \langle \omega, u-u_0 \rangle\quad \text{if}\,\, u\in u_0+ W^{s,2}_0(\Omega)
\end{eqnarray}
and $J_{\omega}(u)=+\infty$ otherwise. We observe that $J_{\omega}$ is strictly convex, lower semicontinuous and coercive and that $J_{\omega}(u_0)=0.$ We remark that the real domain of the functional $J_{\omega}$ is given by
\[
\big\{u\in u_0+W^{s,2}_0(\Omega)\,:\, G_0(x,u-u_0)\in L^1(\Omega)\big\}.
\]
\begin{theorem}\label{thm:equivvariaz} For every $\omega \in W^{-s,2}(\Omega)$ and $u\in W^{s,2}_{\operatorname{loc}}(\Omega)$, it follows that $u$ is the minimum of $J_{\omega}$ if and only if $u$ verifies
\begin{equation}\label{eq:disvar}
\begin{cases}
u>0\,\, \text{a.e. in}\,\,\Omega\,\, \text{and} \,\, u^{-\gamma} \in L^1_{\operatorname{loc}}(\Omega),\\
\\
{\displaystyle \iint_{\mathbb R^{2N}}\frac{\big(u(x)-u(y)\big)\big((v(x)-u(x))-(v(y)-u(y))\big)}{|x-y|^{N+2s}}\,dx\,dy-\int_{\Omega}u^{-\gamma}(v-u)\,dx\geq \langle \omega, v-u \rangle }\\
\,\quad\qquad\qquad\qquad\qquad \qquad\qquad\qquad\qquad\forall v\in u+(W^{s,2}_0(\Omega)\cap L^{\infty}_{c}(\Omega))\,\, \text{with}\,\, v\geq 0 \,\, \text{a.e.}\,\, \text{in}\,\, \Omega,\\
\\
u\leq 0 \,\, \text{on}\,\, \partial \Omega.
\end{cases}
\end{equation}
In particular  for every $\omega \in W^{-s,2}(\Omega)$,  problem \eqref{eq:disvar} has one and only one solution $u\in W^{s,2}_{\operatorname{loc}}(\Omega)$.
\end{theorem}
\begin{proof}
We start proving \eqref{eq:disvar}. Given $\omega \in W^{-s,2}(\Omega)$, using  standard minimization techniques, there exits only one minimum $u\in u_0+  W^{s,2}_0(\Omega)$ of $J_\omega$. Therefore $G_0(x,u-u_0)\in L^1(\Omega)$ and  (see \eqref{eq:sssstuuuu} and \eqref{eq:G0})
\begin{equation}\label{eq:djskhjkwhsjksoyiyouivb}
u\geq0\quad \text{a.e. in} \quad \Omega.
\end{equation}
Let $v\in u_0+  W^{s,2}_0(\Omega)$ be such that
\[G_0(x,v-u_0)\in L^1(\Omega).\]
Then $v\geq 0$ a.e. in $\Omega$ and moreover  $v-u\in W^{s,2}_0(\Omega)$.  Then since $D_sG_0(\cdot, s)$ is non decreasing (see
\eqref{eq:derivG}), for $\xi \in ((v-u_0) \land (u-u_0),  (v-u_0) \lor (u-u_0))$, we deduce
\begin{eqnarray}\nonumber
&&L^1(\Omega)\ni G_0(x,v-u_0) -G_0(x,u-u_0)= (u_0^{-\gamma}-(u_0+\xi)^{-\gamma})(v-u)\\\nonumber
&&\geq (u_0^{-\gamma}-u^{-\gamma})(v-u),
\end{eqnarray}
namely
\begin{eqnarray}\label{eq:dsaljdsnkkajhlj}
(\frac{1}{u_0^{\gamma}}-\frac{1}{u^{\gamma}})(v-u)\in L^1(\Omega).
\end{eqnarray}
Since $G_0(x,\cdot)$ is convex (see \eqref{eq:G0}) we deduce also that, for $t\in [0,1]$,
\[G_0(x,t(v-u_0)+(1-t)(u-u_0))=G_0(x,u-u_0+t(v-u))  \in L^1(\Omega).\]  Since $u$ is the minimum point, for $t\in (0,1]$ we get
\begin{eqnarray}\label{eq:nortonnnn}\\\nonumber
&&0\leq \frac{J_{\omega}(u+t(v-u))-J_{\omega}(u)}{t}\\\nonumber
&&=\frac{c_{N,s}}{2}\iint_{\mathbb{R}^{2N}} \frac{\big((u(x)-u_0(x)) - (u(y)-u_0(y)\big )\, \big((v(x)-u(x)) - (v(y)-u(y))\big)}{|x - y|^{N+2s}}\, dx\, dy\\\nonumber
&&+ t\frac{c_{N,s}}{4}\iint_{\mathbb{R}^{2N}} \frac{\big((v(x)-u(x)) - (v(y)-u(y))\big)^2}{|x - y|^{N+2s}}\, dx\, dy\\\nonumber&&+\frac{1}{t}\Bigg(\int_{\Omega}G_0(x,u-u_0+t(v-u))\, dx
-\int_{\Omega}G_0(x,u-u_0)\, dx\Bigg)-\langle \omega,(v-u)\rangle\\\nonumber
&&=\frac{c_{N,s}}{2}\iint_{\mathbb{R}^{2N}} \frac{\big((u(x)-u_0(x)) - (u(y)-u_0(y)\big )\, \big((v(x)-u(x)) - (v(y)-u(y))\big)}{|x - y|^{N+2s}}\, dx\, dy\\\nonumber
&&+ t\frac{c_{N,s}}{4}\iint_{\mathbb{R}^{2N}} \frac{\big((v(x)-u(x)) - (v(y)-u(y))\big)^2}{|x - y|^{N+2s}}\, dx\, dy
\\\nonumber
&&+\int_{\Omega}\left(\frac{1}{u_0^{\gamma}}-\frac{1}{(u_0+\xi_t)^{\gamma}}\right)(v-u)\,dx
-\langle \omega,(v-u)\rangle,
\end{eqnarray}
with  $\xi_t \in ((u-u_0+t(v-u)) \land (u-u_0),  (u-u_0+t(v-u)) \lor (u-u_0))$.
Recalling \eqref{eq:dsaljdsnkkajhlj} and that $v-u\in W^{s,2}_0(\Omega)$, passing to the limit for $t\rightarrow 0^+$ in \eqref{eq:nortonnnn} we obtain
\begin{eqnarray}\label{eq:nortonnnn1}\\\nonumber
&&\frac{c_{N,s}}{2}\iint_{\mathbb{R}^{2N}} \frac{\big((u(x)-u_0(x)) - (u(y)-u_0(y)\big )\, \big((v(x)-u(x)) - (v(y)-u(y))\big)}{|x - y|^{N+2s}}\, dx\, dy\\\nonumber
&&\geq \int_{\Omega}\left(\frac{1}{u^{\gamma}}-\frac{1}{u_0^{\gamma}}\right)(v-u)\,dx
+\langle \omega,(v-u)\rangle,
\end{eqnarray}
for every $v\in u_0+  W^{s,2}_0(\Omega)$ such that
$G_0(x,v-u_0)\in L^1(\Omega).$ In particular  \eqref{eq:dsaljdsnkkajhlj} holds for all $v\in C^{\infty}_c(\Omega)$ with  $v\geq 0$. Therefore (since $v$ is arbitrary) we obtain that
\[
(\frac{1}{u_0^{\gamma}}-\frac{1}{u^{\gamma}})v\in L^1(\Omega)\qquad\forall  v\in C^{\infty}_c(\Omega)\,\, \text{with}\,\, v\geq0,
\]
whence  $u^{-\gamma}\in L^1_{\operatorname{loc}}(\Omega)$ and (see also \eqref{eq:djskhjkwhsjksoyiyouivb}) $u>0$ a.e. in $\Omega$.
For $\varepsilon, \sigma>0$ let us define
\begin{equation}\label{eq:defpiaga}
v=\min\{u-u_0,\varepsilon-(u_0-\sigma)^+\}.\end{equation}
Since $t\rightarrow t^+$, $t\in \mathbb R$ is a Lipschitz function, we remark that
\begin{equation}\label{eq:assolo}
w(x):=\varepsilon-(u_0-\sigma)^+\in W^{s,2}_{\operatorname{loc}}(\Omega).
\end{equation}
Moreover, by Proposition \ref{pro:u_0}, we know that $u_0\in C(\overline\Omega)$. Therefore there exists a compact set $K\Subset\Omega$ such that $u< \sigma $ in $K^c=\mathbb R^N\setminus K$. We want to show
\begin{equation}\label{eq:kjsakjsmmmhued}\iint_{\mathbb{R}^{2N}} \frac{|w(x) -  w(y)|^2}{|x - y|^{N+2s}}\, dx\, dy< +\infty.\end{equation}
We have
\begin{equation}\label{eq:stb1}
\iint_{\mathbb{R}^{2N}} \frac{|w(x) -  w(y)|^2}{|x - y|^{N+2s}}\, dx\, dy =\iint_{\mathbb{R}^{2N}\setminus(K^c \times K^c)}\frac{|w(x) -  w(y)|^2}{|x-y|^{N+2s}}\,dx\,dy \\\end{equation}
since
\begin{equation}\nonumber
\int_{K^c\times K^c}\frac{|w(x) -  w(y)|^2}{|x-y|^{N+2s}}\,dx\,dy=0.
\end{equation}
By a symmetry argument
\begin{eqnarray}\label{eq:jsakjdkakdasvampun}
&&\iint_{{\mathbb{R}^{2N}\setminus(K^c \times K^c)}} \frac{|w(x) -  w(y)|^2}{|x - y|^{N+2s}}\, dx\, dy= \iint_{K\times K} \frac{|w(x) -  w(y)|^2}{|x - y|^{N+2s}}\, dx\, dy
\\\nonumber
&&+ 2\iint_{K\times K^c} \frac{|w(x) -  w(y)|^2}{|x - y|^{N+2s}}\, dx\, dy
\end{eqnarray}
and readily by \eqref{eq:assolo}
\[
\iint_{K\times K} \frac{|w(x) -  w(y)|^2}{|x - y|^{N+2s}}\, dx\, dy< +\infty.
\]
Let $\delta =\text{dist}(K,\partial \Omega)/2$.  We have
\begin{eqnarray}\label{eq:pentatonicscale}
&&\iint_{K \times K^c}\frac{|w(x) -  w(y)|^2}{|x-y|^{N+2s}}\,dx\,dy\\\nonumber &&=\int_{K}\,dx\int_{K^c\cap\{|y-x|\leq \delta \} }\frac{|w(x) -  w(y)|^2}{|x-y|^{N+2s}}\,dy+\int_{K}\,dx\int_{K^c\cap\{|y-x|\geq\delta\} }\frac{|w(x) -  w(y)|^2}{|x-y|^{N+2s}}\,dy\\\nonumber
&&=I_1+I_2.
\end{eqnarray}
In particular let us consider a compact $\hat K \Subset\Omega$ such that
\[K\subset \hat K \qquad \text{and}\qquad \text{for}\,\, x\in K\,\, \text{fixed}\quad  K^c\cap\{|y-x|< \delta\} \subset  \hat K.\]
Using \eqref{eq:assolo} we deduce that
\begin{eqnarray}\label{eq:eqgian1}
I_1 &\leq&\int_{K}\, dx\int_{K^c\cap\{|y-x|< \delta\} }\frac{|w(x)-w(y)|^2}{|x-y|^{N+2s}}\,dy\\\nonumber
&\leq&  \int_{\hat K}\, dx\int_{\hat K}\frac{|w(x)-w(y)|^2}{|x-y|^{N+2s-2}}\,dy<+\infty.
\end{eqnarray}
On the other hand
\begin{eqnarray}\label{eq:eqgian2}
I_2&\leq&C(\|u_0\|_{L^{\infty}(\Omega)}) \int_{K}\,dx\int_{\mathbb{R}^N\cap\{|y-x|\geq \delta\} }\frac{1}{|x-y|^{N+2s}}\,dy\\\nonumber
&\leq&\int_{K}dx\int_{\mathbb{R}^N\setminus B_{\delta}(0)}\frac{1}{|y|^{N+2s}}\,dy< +\infty.
\end{eqnarray}
Therefore, recalling \eqref{eq:stb1}, \eqref{eq:jsakjdkakdasvampun},  \eqref{eq:eqgian1} and \eqref{eq:eqgian2}, we obtain \eqref{eq:kjsakjsmmmhued}.
By the definition \eqref{eq:defpiaga}, we deduce  that $v=0$ a.e. in $\mathbb R^N\setminus \Omega$ and $v\in L^2(\Omega)$.
Finally by \eqref{eq:kjsakjsmmmhued}  and recalling also  that $u-u_0\in W^{s,2}_0(\Omega)$, we get that $v\in W^{s,2}_0(\Omega)$.
Using \eqref{eq:defpiaga} we infer that either $v=u-u_0$ or $\varepsilon=v\leq u-u_0$ or $v=\varepsilon +\sigma -u_0$ and $u_0\geq \sigma$. In all three cases (see \eqref{eq:G0}) we have that $G_0(x,v)\in L^1(\Omega)$ and that
\begin{equation}\label{eq:vasconumero1}((u_0-\sigma)^++u-u_0-\varepsilon)^+=u-u_0-v\in W_0^{s,2}(\Omega)
\end{equation}
and
\begin{equation}\label{eq:sakslachicco}
(\frac{1}{u_0^{\gamma}}-\frac{1}{u^{\gamma}})(v+u_0-u)\in L^1(\Omega),\end{equation}
where we used a similar argument already used  to get \eqref{eq:dsaljdsnkkajhlj}. Then
we use \eqref{eq:nortonnnn1}, (replacing $v$ with $u_0+v$)
\begin{eqnarray}\label{eq:secondavolta}\\\nonumber
&&\frac{c_{N,s}}{2}\iint_{\mathbb{R}^{2N}} \big((u(x)-u_0(x)) - (u(y)-u_0(y)\big)\cdot\\\nonumber
&&
\cdot\frac{\big((v(x)+u_0(x)-u(x)) - (v(y)+u_0(y)-u(y))\big)}{|x - y|^{N+2s}}\, dx\, dy\\\nonumber
&&\geq \int_{\Omega}\left(\frac{1}{u^{\gamma}}-\frac{1}{u_0^{\gamma}}\right)(v+u_0-u)\,dx
+\langle \omega,(v+u_0-u)\rangle.
\end{eqnarray}
In particular by \eqref{eq:defpiaga}, since $u\neq u_0+v$ implies $u>\varepsilon$, from \eqref{eq:sakslachicco}, we have that both
\begin{equation}\label{eq:dakladlserfb}\frac{1}{u^{\gamma}}(v+u_0-u)\in L^1({\Omega})\qquad \text{and}\qquad \frac{1}{u_0^{\gamma}}(v+u_0-u)\in L^1({\Omega}).\end{equation}
We know that  $u_0$ (see Proposition \ref{main}) satisfies
\begin{equation}\label{eq:chicaaecccasak}
\frac{c_{N,s}}{2}\iint_{\mathbb R^{2N}}\frac{(u_0(x)-u_0(y))(\varphi(x)-\varphi(y))}{|x-y|^{N+2s}}\, dx\, dy= \int_{\Omega}\frac{\varphi}{u_0^\gamma}\, dx,
\end{equation}
for every $\varphi\in C^{\infty}_c(\Omega)$. Using the nonlocal Kato inequality \cite{ChHuVeLa}  we get
\begin{equation}\label{eq:grigperl}
\frac{c_{N,s}}{2}\iint_{\mathbb R^{2N}}\frac{((u_0(x)-\sigma)^+-(u_0(y)-\sigma)^+)(\varphi(x)-\varphi(y))}{|x-y|^{N+2s}}\, dx\, dy\leq \int_{\Omega}\frac{\varphi}{u_0^\gamma}\, dx,
\end{equation}
for every $\varphi\in C^{\infty}_c(\Omega)$, $\varphi\geq 0$.
Using standard arguments, we point out that the inequality \eqref{eq:grigperl} holds true for non negative $\varphi\in W^{s,2}_0(\Omega)$ with compact support contained in $\Omega$.
By density, let $\varphi_n\in C^{\infty}_c(\Omega)$  such that $\varphi_n^+ \rightarrow u-u_0-v$ in $W^{s,2}_0(\Omega)$. Let us define
\begin{equation}\label{eq:testconpartepositiva}
 \tilde\varphi_n:=\text{min}\{u-u_0-v, \varphi_n^+\}. \end{equation}
 As we did above (see \eqref{eq:kjsakjsmmmhued}) we can   deduce that $(u_0-\sigma)^+\in W^{s,2}_0(\Omega)$. Therefore, using~\eqref{eq:grigperl} with $ \tilde\varphi_n$ defined in \eqref{eq:testconpartepositiva}, we  pass to the limit using \eqref{eq:dakladlserfb} and dominate convergence theorem, getting
  \begin{eqnarray}\label{eq:numeroperlimite}
&&\frac{c_{N,s}}{2}\iint_{\mathbb R^{2N}}\big((u_0(x)-\sigma)^+-(u_0(y)-\sigma)^+\big)\cdot
\\\nonumber
&&\cdot\frac{\big((u(x)-u_0(x)-v(x))-(u(y)-u_0(y)-v(y))\big)}{|x-y|^{N+2s}}\, dx\, dy\\\nonumber
&&\leq \int_{\Omega}\frac{u-u_0-v}{u_0^\gamma}\, dx.
\end{eqnarray}
Combining  \eqref{eq:numeroperlimite} with  \eqref{eq:secondavolta} we deduce
\begin{eqnarray}\label{eq:earrrrr}\\\nonumber
&&\frac{c_{N,s}}{2}\iint_{\mathbb R^{2N}}\big(((u_0(x)-\sigma)^++u(x)-u_0(x)-\varepsilon)-((u_0(y)-\sigma)^++u(y)-u_0(y)-\varepsilon)\big)\cdot
\\\nonumber
&&\cdot\frac{\big((u(x)-u_0(x)-v(x))-(u(y)-u_0(y)-v(y))\big)}{|x-y|^{N+2s}}\, dx\, dy\\\nonumber
&&\leq \int_{\Omega}\frac{1}{u^{\gamma}}(u-u_0-v)\,dx+\langle \omega,(u-u_0-v)\rangle.\\\nonumber
&&\leq \varepsilon^{-\gamma}\int_{\Omega}(u-u_0-v)\,dx+\langle \omega,(u-u_0-v)\rangle.
\end{eqnarray}
Let us set $f:=(u_0-\sigma)^++u-u_0-\varepsilon$ and observe that by  \eqref{eq:vasconumero1}, one has that
$f^+=u-u_0-v$.
We have that
\begin{eqnarray}\label{eq:earrrrbascqgkdsr}\\\nonumber
&&\frac{c_{N,s}}{2}\iint_{\mathbb R^{2N}}\frac{(f(x)-f(y))(f^+(x)-f^+(y))}{|x-y|^{N+2s}}\,dx\,dy\\\nonumber
&&=\frac{c_{N,s}}{2}\iint_{\mathbb R^{2N}}\frac{\big((f(x)-f(x)^+)-(f(y)-f(y)^+)\big)(f^+(x)-f^+(y))}{|x-y|^{N+2s}}\,dx\,dy\\\nonumber
&&+ \frac{c_{N,s}}{2}\iint_{\mathbb R^{2N}}\frac{|f^+(x)-f^+(y)|^2}{|x-y|^{N+2s}}\,dx\,dy
\\\nonumber
&&\geq  \frac{c_{N,s}}{2}\iint_{\mathbb R^{2N}}\frac{|f^+(x)-f^+(y)|^2}{|x-y|^{N+2s}}\,dx\,dy,
\end{eqnarray}
where we used the fact that
\[\frac{c_{N,s}}{2}\iint_{\mathbb R^{2N}}\frac{\big((f(x)-f(x)^+)-(f(y)-f(y)^+)\big)(f^+(x)-f^+(y))}{|x-y|^{N+2s}}\,dx\,dy\geq0.\]
In fact let $K_f^+=\text{support}\, (f^+)$ and $K_f^-=\text{support}\, (f^-)$. Therefore (using also a symmetry argument)
\begin{eqnarray}\nonumber
&&\frac{c_{N,s}}{2}\iint_{\mathbb R^{2N}}\frac{\big((f(x)-f(x)^+)-(f(y)-f(y)^+)\big)(f^+(x)-f^+(y))}{|x-y|^{N+2s}}\,dx\,dy
\\\nonumber
&&=c_{N,s}\iint_{K_f^+\times K_f^-}\frac{-f(y)f(x)}{|x-y|^{N+2s}}\,dx\,dy\geq 0
.\end{eqnarray}
Collecting \eqref{eq:earrrrr} and \eqref{eq:earrrrbascqgkdsr} we finally deduce
\begin{eqnarray}\nonumber
&&\frac{c_{N,s}}{2}\iint_{\mathbb R^{2N}}\frac{|(u(x)-u_0(x)-v(x))-(u(y)-u_0(y)-v(y))|^2}{|x-y|^{N+2s}}\,dx\,dy
\\\nonumber
&&\leq \varepsilon^{-\gamma}\int_{\Omega}(u-u_0-v)\,dx+\langle \omega,(u-u_0-v)\rangle.
\end{eqnarray}
Hence for any $\varepsilon>0$ (see also \eqref{eq:vasconumero1}),
\[((u_0-\sigma)^++u-u_0-\varepsilon)^+=u(x)-u_0(x)-v(x)\] is uniformly bounded w.r.t. $\sigma$ in $W^{s,2}_0(\Omega)$.  By Fatou's Lemma, for $\sigma \rightarrow 0^+$ we have that $(u-\varepsilon)^+\in W^{s,2}_0(\Omega)$, that is $u\leq 0$ on $\partial \Omega$ according to Definition \ref{rediri}.

Let now $v\in u+(W^{s,2}_0(\Omega)\cap L_c^\infty(\Omega))$ with $v\geq 0$ a.e. in $\Omega$ and $v_0\in C^{\infty}_c(\Omega)$, $v_0\geq0$ in~$\Omega$ such that $v_0=1$ where $v\neq u$. Then, for any $\varepsilon>0$, $G_0(x, v+\varepsilon v_0-u_0)\in L^1(\Omega)$ and therefore by \eqref{eq:nortonnnn1}
\begin{eqnarray}\nonumber\\\nonumber
&&\frac{c_{N,s}}{2}\iint_{\mathbb{R}^{2N}}((u(x)-u_0(x)) - (u(y)-u_0(y) )\cdot\\\nonumber
&&\cdot\frac{((v(x)+\varepsilon v_0(x)-u(x)) - (v(y)+\varepsilon v_0(y)-u(y)))}{|x - y|^{N+2s}}\, dx\, dy\\\nonumber
&&\geq \int_{\Omega}\left(\frac{1}{u^{\gamma}}-\frac{1}{u_0^{\gamma}}\right)(v+\varepsilon v_0-u)\,dx
+\langle \omega,(v+\varepsilon v_0-u)\rangle,
\end{eqnarray}
namely for $\varepsilon \rightarrow 0$
\begin{eqnarray}\label{eq:c'chidiceno}\\\nonumber
&&\frac{c_{N,s}}{2}\iint_{\mathbb{R}^{2N}}\frac{((u(x)-u_0(x)) - (u(y)-u_0(y) )((v(x)-u(x)) - (v(y)-u(y)))}{|x - y|^{N+2s}}\, dx\, dy\\\nonumber
&&\geq \int_{\Omega}\left(\frac{1}{u^{\gamma}}-\frac{1}{u_0^{\gamma}}\right)(v-u)\,dx
+\langle \omega,(v-u)\rangle.
\end{eqnarray}
By \eqref{eq:chicaaecccasak} we also have that
\[\frac{c_{N,s}}{2}\iint_{\mathbb{R}^{2N}}\frac{(u_0(x) -u_0(y))((v(x)-u(x)) - (v(y)-u(y)))}{|x - y|^{N+2s}}\, dx\, dy= \int_{\Omega}\frac{1}{u_0^{\gamma}}(v-u)\, dx\]
and together with \eqref{eq:c'chidiceno} this gives the second line in problem \eqref{eq:disvar}.

\

Conversely, let $u$ be a solution to \eqref{eq:disvar} and let $\hat u\in W^{s,2}_{\operatorname{loc}}(\Omega)$ be the minimum of the functional $J_{\omega}$. Therefore, as we just proved above, $\hat u$ verifies \eqref{eq:disvar}. Both $u$ and $\hat u$ are sub-supersolution to the problem \eqref{Eq:P}, according to the Definition \eqref{eq:subsupersolunonloc}. Hence by Theorem~\ref{comparison}, it follows that $u\equiv \hat u$, namely $u$ is the minimum of $J_\omega$.
\end{proof}
\begin{proof}[Proof of  Theorem \ref{main}]
If $u$ satisfies \eqref{eq:amkdigola1}, we can use a density argument to show that
\[
\frac{c_{N,s}}{2} \iint_{\mathbb R^{2N}}\frac{\big(u(x)-u(y)\big)(\varphi(x)-\varphi(y))}{|x-y|^{N+2s}}\,dx\,dy-\int_{\Omega}u^{-\gamma}\varphi\,dx= \langle \omega,\varphi\rangle,
\]
for all $\varphi\in W^{s,2}_{0}(\Omega)\cap L_c^{\infty}(\Omega)$. In fact  we can select a sequence $\{\varphi_\varepsilon
\}$ of approximating functions, such that for $\varepsilon$ that goes to zero, we have  $\|\varphi_\varepsilon-\varphi\|_{W^{s,2}_{0}(\Omega)}\rightarrow 0$  and for any $\rho>0$
\[\text{support}\,(\varphi_{\varepsilon})\subseteq \mathcal N_\rho(\text{support}\,(\varphi)), \]
where $\mathcal N_\rho(\text{support}\,\varphi)$ denotes a $\rho$-neighborood of the $\text{support}\,\varphi$. Then we use $\varphi_\varepsilon$ as test function in \eqref{eq:amkdigola1} and  we pass to the limit.

Assume now $\omega\in W^{-s,2}(\Omega)\cap L^1_{\operatorname{loc}}(\Omega)$ and that \eqref{eq:amkdigola2} holds. Obviously for every $v\in C^{\infty}_c(\Omega)$ such that $v\geq 0$ we deduce
\begin{equation}\label{eq:iltuorocknsamn1}
\frac 12 c_{N,s}\iint_{\mathbb R^{2N}}\frac{(u(x)-u(y))(v(x)-v(y))}{|x-y|^{N+2s}}\,dx\,dy-\int_{\Omega}u^{-\gamma}v\,dx\geq \int_{\Omega} \omega v \,dx.
\end{equation}
If $v\in C^{\infty}_c(\Omega)$ with $v\leq 0$, for $t>0$,  let us define $v_t=(u+tv)^+$. Let us denote   $K_{v_t}= \operatorname{supp}(v_t) $, $K_{v_t}^c:= \mathbb R^N \setminus K_{v_t}$ and use the decomposition
\begin{equation}\nonumber
\begin{split}
\mathbb{R}^N\times \mathbb{R}^N\,=\,\left( K_{v_t}\cup K_{v_t}^c\right)\times \left( K_{v_t}\cup K_{v_t}^c\right).
\end{split}
\end{equation}
Thus setting
\begin{equation}\label{eq:skjkakdclauddio}
\varphi_t=(v_t-u)/t,
\end{equation}
we have
\begin{eqnarray}\label{eq:splitgentiloni}
&&\frac{c_{N,s}}{2}\iint_{\mathbb{R}^{2N}}\frac{(u(x)-u(y))(\varphi_t(x)-\varphi_t(y))}{|x-y|^{N+2s}}\,dx\,dy\\\nonumber&&= \frac{c_{N,s}}{2}\iint_{\mathbb{R}^{2N}\setminus(K_{v_t}^c\times K_{v_t}^c)}\frac{(u(x)-u(y))(\varphi_t(x)-\varphi_t(y))}{|x-y|^{N+2s}}\,dx\,dy \\\nonumber
&&-\frac{c_{N,s}}{2t}\iint_{K_{v_t}^c \times K_{v_t}^c}\frac{|u(x)-u(y)|^2}{|x-y|^{N+2s}}\,dx\,dy\\\nonumber
&&
\leq \frac{c_{N,s}}{2}\iint_{\mathbb{R}^{2N}\setminus(K_{v_t}^c\times K_{v_t}^c)}\frac{(u(x)-u(y))(\varphi_t(x)-\varphi_t(y))}{|x-y|^{N+2s}}\,dx\,dy
\\\nonumber
&&
= \frac{c_{N,s}}{2}\iint_{K_{v_t}\times K_{v_t}}\frac{(u(x)-u(y))(v(x)-v(y))}{|x-y|^{N+2s}}\,dx\,dy
\\\nonumber
&&
+{c_{N,s}}\iint_{(K_{v_t}\times K_{v_t}^c)\cap\{u(x)\geq u(y)\}}\frac{(u(x)-u(y))(\varphi_t(x)-\varphi_t(y))}{|x-y|^{N+2s}}\,dx\,dy
\\\nonumber
&&
+{c_{N,s}}\iint_{(K_{v_t}\times K_{v_t}^c)\cap\{u(x)<u(y)\}}\frac{(u(x)-u(y))(\varphi_t(x)-\varphi_t(y))}{|x-y|^{N+2s}}\,dx\,dy:=A_1+A_2+A_3.
\end{eqnarray}
We examine the last  three terms in \eqref{eq:splitgentiloni}.
Using a  similar argument as in equations \eqref{eq:pentatonicscale}, \eqref{eq:eqgian1} and \eqref{eq:eqgian2}, since $u \in W^{s,2}_{\operatorname{loc}}(\Omega)$ and $v\in C^{\infty}_c(\Omega)$, we obtain  that
\begin{eqnarray}\label{eq:kanottedivasco}
&&A_1=\frac{c_{N,s}}{2}\iint_{K_{v_t}\times K_{v_t}}\frac{(u(x)-u(y))(v(x)-v(y))}{|x-y|^{N+2s}}\,dx\,dy\\\nonumber
&&\leq\frac{c_{N,s}}{2}\iint_{\Omega\times \Omega}\frac{|(u(x)-u(y))(v(x)-v(y))|}{|x-y|^{N+2s}}\,dx\,dy<+\infty.
\end{eqnarray}
To get \eqref{eq:kanottedivasco},  we point out that,  since \eqref{eq:amkdigola2} holds, thanks to Theorem \ref{thm:equivvariaz}, we have that $u\in W^{s,2}_{\operatorname{loc}}(\Omega)$ is the minimum of $J_\omega$ defined in $\eqref{eq:ognivoltanottevasco}$. Therefore $u\in u_0+W^{s,2}_0(\Omega)$ and by Proposition \ref{pro:u_0} it follows that $u\in L^1(\Omega).$ From \eqref{eq:kanottedivasco} we deduce also that
\begin{eqnarray}\label{eq:kanottedivascograziano}
&&\frac{(u(x)-u(y))(v(x)-v(y))}{|x-y|^{N+2s}}\cdot \chi_{K_{v_t}\times K_{v_t}}(x,y) \\\nonumber
&&\leq \frac{|(u(x)-u(y))(v(x)-v(y))|}{|x-y|^{N+2s}}\in L^1(\Omega \times \Omega),
\end{eqnarray}
where by $\chi_{\mathcal A}$ we denote the characteristic function of a set $\mathcal A$.
Using the definition \eqref{eq:skjkakdclauddio}, we infer that
\begin{eqnarray}\label{eq:h32dsalkehfqwehfwohiuhu}
&&A_2={c_{N,s}}\iint_{(K_{v_t}\times K_{v_t}^c)\cap\{u(x)\geq u(y)\}}\frac{(u(x)-u(y))(v(x)+u(y)/t)}{|x-y|^{N+2s}}\,dx\,dy\\\nonumber
&&\leq {c_{N,s}}\iint_{(K_{v_t}\times K_{v_t}^c)\cap\{u(x)\geq u(y)\}}\frac{(u(x)-u(y))(v(x)-v(y))}{|x-y|^{N+2s}}\,dx\,dy\\\nonumber
&&\leq {c_{N,s}}\iint_{\Omega\times \mathbb R^N\setminus \Omega}\frac{|(u(x)-u(y)||v(x)-v(y)|}{|x-y|^{N+2s}}\,dx\,dy\\\nonumber
&&+c_{N,s}\iint_{\Omega\times  \Omega}\frac{|(u(x)-u(y)||v(x)-v(y)|}{|x-y|^{N+2s}}\,dx\,dy.\end{eqnarray}
Therefore we have
\begin{eqnarray}\nonumber
&&\iint_{\Omega\times \mathbb R^N\setminus \Omega}\frac{|(u(x)-u(y)||v(x)-v(y)|}{|x-y|^{N+2s}}\,dx\,dy\\\nonumber
&&\leq C(s,N,\Omega)\int_{\Omega}|u(x)||v(x)|\,dx\int _{|y|\geq \bar R}\frac{1}{{|y|^{N+2s}}}\, dy<+\infty,
\end{eqnarray}
where we used the fact that $u(x)=v(x)=0$ for a.e. $x \in \mathbb R^N\setminus \Omega$ and $\text{dist}(\partial K_v,\partial \Omega)=\bar R$, since $v$ has compact support contained in $\Omega$ and $u\in L^1(\Omega)$.  Arguing as in \eqref{eq:kanottedivasco}, we have
\begin{eqnarray}\nonumber
c_{N,s}\iint_{\Omega\times \Omega}\frac{|(u(x)-u(y)||v(x)-v(y)|}{|x-y|^{N+2s}}\,dx\,dy<+\infty.
\end{eqnarray}
Hence, from \eqref{eq:h32dsalkehfqwehfwohiuhu} we deduce that
\begin{equation}\label{eq:h32dsalkehfqwehfwohiuhu1}
A_2\leq {c_{N,s}}\iint_{\Omega\times \mathbb R^N}\frac{|(u(x)-u(y)||v(x)-v(y)|}{|x-y|^{N+2s}}\,dx\,dy<+\infty.
\end{equation}
Actually we deduce that
\begin{eqnarray}\label{eq:kanottedivascograzianoluigi}
&&\frac{(u(x)-u(y))(\varphi_t(x)-\varphi_t(y))}{|x-y|^{N+2s}}\cdot \chi_{(K_{v_t}\times K_{v_t}^c)\cap\{u(x)\geq u(y)\}}
\\\nonumber
&&\leq\frac{|(u(x)-u(y)||v(x)-v(y)|}{|x-y|^{N+2s}}\in L^1(\Omega\times \mathbb R^N).
\end{eqnarray}
By the definition \eqref{eq:skjkakdclauddio} we also  get
\begin{eqnarray}\label{eq:laterzavasc}
&&A_3={c_{N,s}}\iint_{(K_{v_t}\times K_{v_t}^c)\cap\{u(x)<u(y)\}}\frac{(u(x)-u(y))(v(x)+u(y)/t)}{|x-y|^{N+2s}}\,dx\,dy\\\nonumber
&&\leq- \frac{c_{N,s}}{t}\iint_{(K_{v_t}\times K_{v_t}^c)\cap\{u(x)<u(y)\}}\frac{|u(x)-u(y)|^2}{|x-y|^{N+2s}}\,dx\,dy\leq 0.
\end{eqnarray}
Using \eqref{eq:splitgentiloni} and \eqref{eq:laterzavasc} we deduce
\begin{eqnarray}\nonumber
&&\frac{c_{N,s}}{2}\iint_{K_{v_t}\times K_{v_t}}\frac{(u(x)-u(y))(v(x)-v(y))}{|x-y|^{N+2s}}\,dx\,dy
\\\nonumber
&&
+{c_{N,s}}\iint_{(K_{v_t}\times K_{v_t}^c)\cap\{u(x)\geq u(y)\}}\frac{(u(x)-u(y))(\varphi_t(x)-\varphi_t(y))}{|x-y|^{N+2s}}\,dx\,dy\\\nonumber
&&\geq\frac{c_{N,s}}{2}\iint_{\mathbb{R}^{2N}}\frac{(u(x)-u(y))(\varphi_t(x)-\varphi_t(y))}{|x-y|^{N+2s}}\,dx\,dy.
\end{eqnarray}
Observe that $|\varphi_t|\leq |v|$. Since  \eqref{eq:amkdigola2} holds, we infer that
\begin{eqnarray}\label{eq:ioetenotvasc}
&&\frac{c_{N,s}}{2}\iint_{K_{v_t}\times K_{v_t}}\frac{(u(x)-u(y))(v(x)-v(y))}{|x-y|^{N+2s}}\,dx\,dy
\\\nonumber
&&
+{c_{N,s}}\iint_{(K_{v_t}\times K_{v_t}^c)\cap\{u(x)\geq u(y)\}}\frac{(u(x)-u(y))(\varphi_t(x)-\varphi_t(y))}{|x-y|^{N+2s}}\,dx\,dy
\\\nonumber
&&\geq \int_{\Omega}u^{-\gamma}\varphi_t\,dx+ \int_{\Omega} \omega \varphi_t \,dx.\end{eqnarray}
Recalling \eqref{eq:kanottedivasco}, \eqref{eq:kanottedivascograziano}, \eqref{eq:h32dsalkehfqwehfwohiuhu}, \eqref{eq:h32dsalkehfqwehfwohiuhu1}, \eqref{eq:kanottedivascograzianoluigi} and  that $u>0$ a.e. in $\Omega$, using the dominate convergence theorem in \eqref{eq:ioetenotvasc}, we finally get
\begin{eqnarray}\label{eq:ioetenotvasc222}
&&\frac{c_{N,s}}{2}\iint_{\Omega\times \Omega}\frac{(u(x)-u(y))(v(x)-v(y))}{|x-y|^{N+2s}}\,dx\,dy
\\\nonumber
&&
+{c_{N,s}}\iint_{(\Omega\times \mathbb {R}^N\setminus \Omega)}\frac{(u(x)-u(y))(v(x)-v(y))}{|x-y|^{N+2s}}\,dx\,dy
\\\nonumber
&&\geq \int_{\Omega}u^{-\gamma}v\,dx+ \int_{\Omega} \omega v \,dx.\end{eqnarray}
Up to a change of  variables in the second integrale in the l.h.s of \eqref{eq:ioetenotvasc222}, we deduce
\begin{equation}\label{eq:tuttoeposiibileunmondo}
\frac{c_{N,s}}{2}\iint_{\mathbb R^N\times \mathbb R^N}\frac{(u(x)-u(y))(v(x)-v(y))}{|x-y|^{N+2s}}\,dx\,dy\geq \int_{\Omega}u^{-\gamma}v\,dx+ \int_{\Omega} \omega v \,dx,\end{equation}
for all $v\in C^{\infty}_c(\Omega)$ with $v\leq 0$. Thanks to equations \eqref{eq:iltuorocknsamn1} and \eqref{eq:tuttoeposiibileunmondo} we deduce that $u$ satisfies \eqref{eq:amkdigola1}, concluding the proof.
\end{proof}
\begin{proof}[Proof of  Theorem \ref{eq:thspostato}]
Let $u\in W^{s,2}_{\operatorname{loc}}(\Omega)\cap L^{\frac{2N}{N-2s}}(\Omega)$ such that  \eqref{eq:amkdigola11estioam} holds. Let $\omega=g(x,u)=g_1(x,u-u_0)$. Therefore $\omega\in W^{-s,2}(\Omega)$. By Theorem \ref{main} and Theorem \ref{thm:equivvariaz} we have that $u\in u_0+W^{s,2}_0(\Omega)$ and $u-u_0$ minimizes \eqref{eq:ognivoltanottevasco}, i.e. for all $v\in W^{s,2}_0(\Omega)$ we have
\begin{eqnarray}\label{eq:vascocomequestasera}
&&\frac{c_{N,s}}{4}\iint_{\mathbb{R}^{2N}} \frac{|v(x) - v(y)|^2}{|x - y|^{N+2s}}\,dx\,dy+\int_{\Omega}G_0(x,v)\, dx\\\nonumber
&&\geq \frac{c_{N,s}}{4}\iint_{\mathbb{R}^{2N}} \frac{\big ((u(x)-u_0(x)) - (u(y)-u_0(y))\big)^2}{|x - y|^{N+2s}}\,dx\,dy+\int_{\Omega}G_0(x,u-u_0)\, dx\\\nonumber
&&-\langle \Phi'(u-u_0), v-(u-u_0)\rangle,
\end{eqnarray}
that is
\[ \langle \Phi'(u-u_0),v-(u-u_0)\rangle+ \Psi(v)-\Psi(u-u_0)\geq 0.\]
Recalling \eqref{eq:miseravallafinemannaia}, $u-u_0$ is a critical point of $F$ in the sense of \cite{Szul}.

Let us assume that \eqref{eq:lastappl} holds. Then we have \eqref{eq:vascocomequestasera}.  From \eqref{eq:lastappl} and Proposition \ref{pro:u_0} we deduce that $u\in W^{s,2}_{\operatorname{loc}}(\Omega)\cap L^{\frac{2N}{N-2s}}(\Omega)$ and therefore $\omega=g(x,u)=g_1(x,u-u_0)\in  W^{-s,2}(\Omega)\cap L^1_{\operatorname{loc}}(\Omega)$. By Theorem \ref{main} and Theorem \ref{thm:equivvariaz} we deduce that $u$ is a solution to~ \eqref{eq:amkdigola11estioam}.
\end{proof}

\end{document}